\documentclass[12pt]{amsart}
\usepackage{amssymb,latexsym,amsmath,amscd,amsthm,graphicx, color}
\usepackage[all]{xy}
\usepackage{pgf,tikz}
\usepackage{mathrsfs}
\usepackage{breqn}
\usepackage{amsmath}
\usepackage{fancyhdr}
\usepackage[numbers]{natbib}
\usepackage[english]{babel}
\usepackage[autostyle]{csquotes}
\usepackage{verbatim}
\DeclareFontFamily{U}{mathx}{\hyphenchar\font45}
\DeclareFontShape{U}{mathx}{m}{n}{
      <5> <6> <7> <8> <9> <10>
      <10.95> <12> <14.4> <17.28> <20.74> <24.88>
      mathx10
      }{}
\DeclareSymbolFont{mathx}{U}{mathx}{m}{n}
\DeclareFontSubstitution{U}{mathx}{m}{n}
\DeclareMathAccent{\widecheck}{0}{mathx}{"71}
\DeclareMathAccent{\wideparen}{0}{mathx}{"75}

\usetikzlibrary{arrows}
\definecolor{uuuuuu}{rgb}{0.26666666666666666,0.26666666666666666,0.26666666666666666}
\definecolor{xdxdff}{rgb}{0.49019607843137253,0.49019607843137253,1.}
\definecolor{ffqqqq}{rgb}{1.,0.,0.}

\newtheorem{theorem}{Theorem}[section]

\theoremstyle{definition}
\newtheorem{remark}[subsection]{Remark}
\newtheorem{definition}[subsection]{Definition}

\theoremstyle{remark}

\newtheorem{proposition}[subsection]{Proposition}
\numberwithin{equation}{section}
\def\ni{\noindent}

\title[Optional running title]{Actual title}

\begin{document}
\title[Dilation Operators in  Besov spaces over Local Fields]{Dilation Operators in  Besov spaces over Local Fields}

%
\author{Salman Ashraf, Qaiser Jahan}
\address{School of Basic Sciences, Indian Institute of Technology Mandi\\ Kamand (H.P.) - 175005, India}
\email{ashrafsalman869@gmail.com; qaiser@iitmandi.ac.in (Email of corresponding author)}







\begin{abstract}
We consider a dilation operator on Besov spaces $(B^s_{r,t}(K))$ over local fields and estimate an operator norm on such a field for $s > \sigma_r = \text{max}\big(\frac{1}{r} -1,~0\big)$ which depends on the constant $k$ unlike the case of Euclidean spaces. In $\mathbb{R}^n$, it is independent of constant. A constant $k$ appears for liming case $s=0$ and $s=\sigma_r$. In case of local fields, the limig case is still open. Further we also estimate the localization property of Besov spaces over local fields.

\end{abstract}

\maketitle



\section{\textbf{Introduction}} 

Besov spaces $B^s_{pq}(\mathbb{R}^n)$ and Triebel-Lizorkin spaces $F^s_{pq}(\mathbb{R}^n)$ with $s \in \mathbb{R},~ 0 < p \leq \infty~(p < \infty~\text{for the F-type spaces}),~0 < q \leq \infty$  on the Euclidean spaces $\mathbb{R}^n$ are usually considered to be two very general scales of function spaces. These two spaces $B^s_{pq}(\mathbb{R}^n)$ and $F^s_{pq}(\mathbb{R}^n)$ contain well-known classical function spaces as special cases, such as  Lebesgue spaces $L_p = F^0_{p,2}(\mathbb{R}^n) $ (with $1 < p < \infty$), Sobolev spaces $W^m_p = F^m_{p,2}(\mathbb{R}^n) $ (with $m \in \mathbb{N},~ 1 < p < \infty$), H\"older spaces $C^s = B^s_{\infty,\infty}(\mathbb{R}^n)$ (with $0 < s  \neq \mathbb{N}$). These function spaces have been extensively studied by H. Triebel, we refer \cite{book_impulse7,book_impulse8} for more details. The theory of partial differential equation is one of the main applications of these two spaces. Dilation operator has been investigated on many function spaces, including the Besov spaces. We state the following results summarizing the behavior of dilation operator on Besov spaces.

In \citep[3.4]{book_impulse7}, Triebel considered the dilation operators of the form  
\begin{equation} \label{r1}
(T_kf)(x) = f(2^kx),\qquad x \in \mathbb{R}^n ,\quad k \in \mathbb{N},
\end{equation}

\ni and showed that for  $0 < p,q \leq \infty$, and $\infty > s > \sigma_p = \text{max}\big(n\big(\frac{1}{p} -1\big),~0\big)$, $T_k$ represent bounded operators from $B^s_{pq}(\mathbb{R}^n)$ into itself and in this case we have
\begin{equation} \label{r2}
\|T_kf ~ \big| ~B^s_{pq}(\mathbb{R}^n) \| \leq c  2^{k(s-\frac{n}{p})} \|f ~\big|~  B^s_{pq}(\mathbb{R}^n) \|. 
\end{equation}
 
Boundedness of $T_k$ for  the limiting case $s= 0$ and $s=\sigma_p$ remained open until Vyb$\acute{ı}$ral in \cite{VJ} and Schneider in \cite{SC}, respectively, gave the final answer and showed that the additional constant $k$ will appear in the calculation of the norm of dilation operator $T_k$. The norm calculated by Vyb$\acute{ı}$ral in \cite{VJ} is as follows.

For $0 < q \leq \infty ,$
\begin{equation} \label{r3}
\|T_k ~ \big| ~\mathcal{L}(B^0_{pq}(\mathbb{R}^n)) \|   \sim   2^{-k\frac{n}{p}} \begin{cases} k^{\frac{1}{q}-\frac{1}{\text{max}(p,q,2)}},~~~\text{if},~~~1 < p < \infty \\
k^{\frac{1}{q}},~~~~\text{if}~~p=1~~~ \text{or}~~~ p=\infty,
\end{cases}
\end{equation}
and by Schneider \cite{SC} is given by,

\begin{equation} \label{r4}
\|T_k ~ \big| ~\mathcal{L}(B^{\sigma_p}_{pq}(\mathbb{R}^n)) \|   \sim   2^{k(\sigma_p-\frac{n}{p})}k^{\frac{1}{q}}, ~~~\text{for}~0 < p \leq 1~ \text{and} ~0 < q \leq \infty,
\end{equation}

\ni where $\|T_k ~ \big| ~\mathcal{L}(B^{s}_{pq}(\mathbb{R}^n)) \|$ denotes the norm of the operator $T_k$ from $B^{s}_{pq}(\mathbb{R}^n)$ into itself.

The boundedness of the dilation operators are studied not only for theoretical aspect, but also it has major applications to several classical problems in analysis. For instance, they appear in the localization of $B^{s}_{pq} $ and $F^{s}_{pq}$ spaces   and H\"older inequalities and sharp embeddings of $ B^ s_ {pq}$ and $F^ s_ {pq} $ spaces (see \citep[chap 2]{book_impulse2}, \cite{book_impulse11}, \cite{book_impulse12}).

The main aim of the present paper is to study dilation operators in the framework of Besov spaces on local fields. In particular, we extend the results obtained by Triebel in \citep[3.4]{book_impulse7}. A local field $K$ is a locally compact, totally disconnected, non-Archimedian norm valued and non-discrete topological field, we refer \cite{book_impulse6} to basic Fourier analysis on local fields. The local fields are essentially of two types (excluding the connected local fields $\mathbb{R}$ and $\mathbb{C}$) namely characteristic zero and of positive characteristic. The characteristic zero local fields include the $p$-adic field $\mathbb{Q}_p$ and the examples of positive characteristic are the Cantor dyadic groups, Vilenkin $p$-groups and $p$-series fields.

Besov spaces $B^s_{rt}(K)$ and Triebel-Lizorkin spaces $F^s_{rt}(K)$ on local fields were introduced and studied by Onneweer and Su \cite{Os}. In \cite{book_impulse4}, Weiyi Su, introduced, \textquote{$p$-type smoothness} of the functions defined on local fields. To complete the theoretical base of the Function spaces on local fields and to broaden the range of its applications, a series of studies have been carried out such as the construction theory of functions on local fields \cite{book_impulse10}, the Weierstrass functions, Cantor functions and their $p$-adic derivatives in local fields, the Lipschitz classes on local fields etc. (see \cite{book_impulse5}). Wavelet theory on local fields has also been developed widely. Jiang, Li and Jin \cite{JLJ} have introduced the concept of multiresolution analysis and wavelet frames on local fields \cite{JL}. Later, Behera and Jahan have developed the theory of wavelets on such a field in a series of papers \cite{BJ1, BJ2, BJ3}.

This article is organized as follows. In Section 2, we provide a brief introduction to local fields and  test function class and distribution on such a field. We have also provided the definition of $p$-type derivatives and Besov spaces $B^s_{rt}(K)$ on local fields. In section 3, we study the dilation operators
\begin{equation} \label{eq1}
(T_kf)(x) = f(\mathfrak{p}^{-k}x),\qquad x \in K ,\quad k \in \mathbb{N},
\end{equation}

\ni in the framework of Besov spaces $B^s_{rt}(K)$ and more precisely, we shall prove the following result:

\begin{theorem} Let $0 < r ,t \leq \infty $, $ s > \sigma_r = \text{max}\big(\frac{1}{r} -1 , 0\big), ~k \in \mathbb{N}$ and $T_k$ defined by (\ref{eq1}). Then 
\begin{align*}
\|T_kf ~ \big| ~{ B_{rt}^s}(K) \| \leq c k^{1/t} q^{k(s-\frac{1}{r})} \|f ~\big|~ { B_{rt}^s}(K) \|,
\end{align*}

\ni for some, $c$ which is independent of $k$ and for all $f \in  B_{rt}^s(K).$
\end{theorem}

Note that, in the case of $B_{rt}^s(K)~(\text{with}~s > \sigma_r),$ Besov norm of operators $T_k$ (defined by (\ref{eq1})) depends on $k$ by the constant $k^{1/t},$ whereas, in the case of $B^s_{pq}(\mathbb{R}^n)$, Besov norm of operators $T_k$ (defined by (\ref{r1}))  is independent of the constant $k$ for $s>\sigma_r$, see \cite{book_impulse7}, in the case of Euclidean spaces it depends on $k$ in the limiting case $s= 0$ and $s=\sigma_p$, see \cite{book_impulse11} for more details. For local fields, the limiting case is still open.

In section 4, we have established the localization property for $ B_{rt}^s(K).$ Construction of type (\ref{1}), involves a compactly supported distribution which is dilated and translated. Such type of localization property for Besov spaces on $\mathbb{R}^n$ is well known; see \cite{book_impulse8,book_impulse11}.

\section{\textbf{Preliminaries}}

\subsection{Local Fields}
Let $K$ be a field and a topological space. If the additive group $K^{+}$ and multiplicative group $K^{*}$ of $K$ both are locally compact Abelian group, then $K$ is called \textit{locally compact topological field}, or \textit{local field.}

If $K$ is any field and is endowed with the discrete topology, then $K$ is a local field. Further, if $K$ is connected, then $K$ is either $\mathbb{R}$ or $\mathbb{C}$. If $K$ is not connected, then it is totally disconnected. So, here local fields means a field $K$ which is locally compact, non-discrete and totally disconnected. If $K$ is of characteristic zero, then K is either a $p$-adic field for some prime number $p$ or a finite algebraic extension of such a field. If $K$ is of positive characteristic, then $K$ is either a field of formal Laurent series over a finite field of characteristic $p$ or an algebraic extension of such a field.

Let $K$ be a local field. Since $K^{+}$ is a locally compact abelian group, we choose a Haar measure $dx$ for $K^{+}$. If $ \alpha \neq 0,~ \alpha \in K$, then $d(\alpha x)$ is also a Haar measure and by the uniqueness of Haar measure we have $d(\alpha x)= |\alpha|dx$. We call $|\alpha|$ the \textit{absolute value} or \textit{valuation} of $\alpha$ which is, in fact, a natural non-Archimedian norm on $K$. A mapping $|\cdot|~ : K \rightarrow \mathbb{R}$ satisfies:

\begin{itemize}
\item[(a)] $|x|=0$ if and only if $x=0$;
\item[(b)] $|xy|= |x||y|$ for all $x,y \in K$;
\item[(c)] $|x+y| \leq \text{max} \{|x|, |y|\},$  for all $x,y \in K$, \
\end{itemize}

The property $(c)$ is called the \textit{ultrametric inequality}. It follows that
\begin{equation*}
|x+y| = \text{max} \{|x|, |y|\}~ \text{if}~ |x| \neq |y|.
\end{equation*}

The set $\mathfrak{D}= \{x \in K : |x| \leq 1\}$ is called the ring of integers in $K$. It is the unique maximal compact subring of $K$. The set $\mathfrak{P}=\{x \in K : |x| < 1\}$ is called the prime ideal in $K$. It is the unique maximal ideal in $\mathfrak{D}$.  It can be proved that $\mathfrak{D}$ is compact and open and hence $\mathfrak{P}$ is also compact and open. Therefore, the residue space $\mathfrak{D}/\mathfrak{P}$ is isomorphic to a finite field $GF(q)$, where $q =p^c$ for some prime number $p$ and $c \in \mathbb{N}$.

Let $\mathfrak{p}$ be a fixed element of maximum absolute value in, $\mathfrak{P}$ and it is called a \textit{prime element} of $K$. For a measurable subset $E$ of $K$, let $|E|=\int_K\bold{1}_E(x)dx$, where $\bold{1}$ is the indicator function of $E$ and $dx$ is the Haar measure of $K$ normalized so that $|\mathfrak{D}|=1$. Then it is easy to prove that $|\mathfrak{P}| = q^{-1}$ and $|\mathfrak{p}| = q^{-1}$ (see~\cite{book_impulse6}).

Let $\mathcal{U}=\{a_i\}_{i=0}^{q-1}$ be any fixed full set of coset representative of $\mathfrak{P}$ in, $\mathfrak{D}$ then each $x \in K$ can be written uniquely as
$x =\sum\limits_{l=k}^{\infty}c_l\mathfrak{p}^k$, where $c_l\in\mathcal{U}$.
If $x (\neq 0) \in K$ then $|x| =q^k$ for some $ k \in \mathbb{Z}.$

Let $\mathfrak{D}^{*} = \mathfrak{D}\setminus \mathfrak{P} = \{x \in \mathfrak{D} : |x|=1\}$; $\mathfrak{D}^*$ is the group of units in $K^*$. If $x \neq 0$, we can write as $x = \mathfrak{p}^k x^\prime$, with $ x^{\prime}\in  \mathfrak{D}^*$. Let $\mathfrak{P}^k = \mathfrak{p}^k \mathfrak{D} = \{x \in K : |x| \leq q^{-k}\}~, k \in \mathbb{Z}$. These are called \textit{fractional ideals}.

The set $\{\mathfrak{P}^k \subset K: k \in \mathbb{Z}\}$ satisfies the following:

\begin{itemize}
\item[(\romannumeral 1)] $\{\mathfrak{P}^k \subset K : k \in \mathbb{Z}\}$ is a base for neighborhood system of identity in $K$, and $\mathfrak{P}^{k+1} \subset \mathfrak{P}^k,~ k \in \mathbb{Z};$
\item[(\romannumeral 2)] $\mathfrak{P}^k ,~ k \in \mathbb{Z},$ is open, closed and compact in $K$;
\item[(\romannumeral 3)] $K = \bigcup\limits_{{k=-\infty}}^{+\infty} \mathfrak{P}^k~~\text{and}~~\{0\}=\bigcap\limits_{{k=-\infty}}^{+\infty} \mathfrak{P}^k.$
\end{itemize} 

Let $\Gamma$ be the character group of the additive group $K^{+}$ of $K$. There is a nontrivial character $\chi \in \Gamma$ which is trivial on $\mathfrak{D}$ but is non-trivial on $\mathfrak{P}^{-1}$. Let $\chi$ be a non-trivial character on $K^{+}$, then the corresponding relationship $\lambda \in K \longleftrightarrow \chi_{\lambda} \in \Gamma$ is determined by $\chi_{\lambda}(x)=\chi(\lambda x),$ and the topological isomorphism is established for $K$ and $\Gamma$, moreover, we have $\Gamma=\{ \chi_{\lambda} : \lambda \in K \}.$

For $k \in \mathbb{Z}$, let $\Gamma^{k}$ be the annihilator of $\mathfrak{P}^{k},$ that is,

\begin{align*}
 \Gamma^{k} = \{\chi \in \Gamma : \forall x \in \mathfrak{P}^{k} \implies \chi(x)=1\}.
\end{align*}

The locally compact group $\Gamma=\{ \chi_{\lambda} : \lambda \in K \}$ can be endowed with the non-Archimedian norm as $ \Gamma^{k}=\{ \chi_{\lambda} \in \Gamma : |\lambda| \leq q^{k}\}=\mathfrak{p}^{-k}\mathfrak{D},~k \in \mathbb{Z}.$

For the Haar measure $dx$ of $K^{+}$, let $d\xi$ be the Haar measure on $\Gamma$, chosen such that 

$$|\Gamma^{0}|=|\mathfrak{D}|=1, ~\text{and}~|\Gamma^{k}| = \dfrac{1}{|\mathfrak{P}^{k}|} = q^{k}.$$

The set $\{\Gamma^k \subset \Gamma :~ k \in \mathbb{Z}\}$ satisfies the following properties:

\begin{itemize}
\item[(\romannumeral 1)] $\Gamma^{k} \subset \Gamma^{k+1},~ k \in \mathbb{Z}$, are increasing sequence in $\Gamma$ which is open, closed and compact;
\item[(\romannumeral 2)] $\Gamma = \bigcup\limits_{{k=-\infty}}^{+\infty} \Gamma^k~~\text{and}~~\{I\}=\bigcap\limits_{{k=-\infty}}^{+\infty} \Gamma^k,$ where $I$ is the unit element of $\Gamma$;
\item[(\romannumeral 3)] $\{\Gamma^k\}_{k\in \mathbb{Z}}$ is a base of the unit $I$ of character group $\Gamma$.
\end{itemize}

\subsection{Function classes}

In this subsection we will introduce some function classes:

\textit{Test function class $S(K)$:} It is the linear space in which functions have the form

$$ \phi(x)= \sum_{j=1}^{n} c_j  \Phi_{\mathfrak{P}^{j}}(x-h_j),~c_j \in \mathbb{C}, h_j \in K, j \in \mathbb{Z},n \in \mathbb{N},$$ 

where $\Phi_{\mathfrak{P}^{j}}$ is the characteristic function of $\mathfrak{P}^{j}.$ The space $S(K)$ is an algebra of continuous functions with compact support that separates points. Consequently, $S(K)$ is dense in $C_0(K)$ as well as $L^r(K),~1 \leq r < \infty.$ Similarly, the test function class $S(\Gamma)$ on $\Gamma$ can be defined. However, since $K$ is isomorphic to $\Gamma$, so $S(K)$ and $S(\Gamma)$ can be regarded as equivalent with respect to absolute value or valuation.

The space $S(K)$ is equipped with a topology as a topological vector space as follows: Define a null sequence in $S(K)$ as a sequence $\{\phi_n\}$ of functions on $S(K)$ in such a way that each $\phi_n$ is constant on cosets of $\mathfrak{P}^l$ and is supported on $\mathfrak{P}^k$ for a fixed pair of integers $k$ and $l$ and the sequence converges to zero uniformly. The space $S(K)$ is complete and separable and is called the \emph{space of testing function}.

Since $S(K)$ is dense in $L^1(K)$, thus the Fourier transformation of $\phi(x) \in S(K)$ is defined by

$$\widehat{\phi}(\xi) \equiv (\mathcal{F}\phi)(\xi) = \int_{K} \phi(x)\bar{\chi}_\xi(x)dx,~~~  \xi  \in \Gamma,$$

and the inverse Fourier transformation of $\phi \in S(K)$ is defined by the formula 

$$ \widecheck{\phi}(x) \equiv (\mathcal{F}^{-1}\phi)(x) = \int_{\Gamma} \phi(\xi)\chi_x(\xi)d\xi,~~~  x  \in K.$$

$S^{\prime}(K)$, the \textit{space of distributions}, is a collection of continuous linear functional on $S(K)$. $S^{\prime}(K)$ is also a complete topological linear space. The action of $f$ in $S^{\prime}(K)$ on an element $\phi$ in $S(K)$ is denoted by $\langle f, \phi \rangle$. The distribution space $S^{\prime}(K)$ is given the weak$\ast$ topology. Convergence in $S^{\prime}(K)$ is defined in the following way: $f_k$ converges to $f$ in $S^{\prime}(K)$ if  $\langle f_k, \phi \rangle$ converges to $\langle f,\phi\rangle$ for any $\phi \in S(K).$

The Fourier transformation $\widehat{f}$ of a distribution  $f \in S^{\prime}(K)$ is defined by

$$ \langle \widehat{f}, \phi \rangle ~ = ~ \langle f, \widehat{\phi} \rangle~~~\text{for~all}~ \phi \in S(K),$$ 

 the inverse Fourier transformation $\widecheck{g}$ is defined by 

$$ \langle \widecheck{g}, \psi \rangle = \langle g, \widecheck{\psi} \rangle~~~\text{for~all}~ \psi \in S(K).$$

\subsection{Derivatives and Integrals} 
In \cite{book_impulse4}, Weiyi Su introduced $p$-type derivatives and $p$-type integrals for general locally
compact Vilenkin group using pseudo-differential operators, which also includes derivatives and integrals of fractional orders.

\begin{definition} 
Let $\alpha > 0$, if for a complex Haar measurable function $f : K \rightarrow \mathbb{C}$, the integral

$$T_{\langle\cdot\rangle^\alpha} f(x) \equiv \int_{\Gamma} \bigg\{\int_{K} \langle \xi \rangle^\alpha f(t) \bar{\chi}_\xi(t-x)dt \bigg\} d\xi$$

\ni exists at $x \in K$, then $T_{\langle\cdot\rangle^\alpha} f(x)$ is said to be a point-wise $\alpha$-order $p$-type derivative of $f(x)$ at $x$, and it is denoted by $f^{\langle\alpha\rangle}(x).$
\end{definition}

\begin{definition} Let $\alpha > 0$, if for a complex Haar measurable function $f : K \rightarrow \mathbb{C}$, the integral

$$T_{\langle\cdot\rangle^{-\alpha}} f(x) \equiv \int_{\Gamma} \bigg\{\int_{K} \langle \xi \rangle^{-\alpha} f(t) \bar{\chi}_\xi(t-x)dt \bigg\} d\xi$$

\ni exists at $x \in K$, then $T_{\langle\cdot\rangle^{-\alpha}} f(x)$ is said to be a point-wise $\alpha$-order $p$-type integral of $f(x)$ at $x$, denoted by $f_{\langle\alpha\rangle}(x).$
\end{definition}

\begin{remark} We suppose that $ \alpha > 0$ in the above definitions, thus the order $\alpha$ of $p$-type derivatives and integrals can be any positive real numbers, thus, fractional order derivatives and fractional order integrals all are contained. Moreover, for $\alpha =0$, 
$$f^{\langle 0 \rangle}(x) = f(x) = f_{\langle 0 \rangle}(x)~~\text{for}~x\in K.$$
\end{remark}

\subsection{Function spaces on Local fields}

\begin{definition} \label{def2} 
 Let $S(K)$ and $S(\Gamma)$ be the test function classes on $K$ and on $\Gamma$, respectively, $S^{\prime}(K)$ and $S^{\prime}(\Gamma)$ are their distribution spaces, respectively.

Take a sequence $\{\phi_j(\xi)\}_{j=0}^{+\infty} \subset S(\Gamma)$ satisfying 

\begin{itemize}
\item[(\romannumeral 1)]   
$\text{supp}~\phi_0 \subset \{\xi \in \Gamma  : |\xi|< q\}$,\\
$\text{supp}~\phi_j \subset \{\xi \in \Gamma : q^{j-1} < |\xi|< q^{j+1}\},~~~~j \in \mathbb{N}$
 
\item[(\romannumeral 2)] 
$\sum_{j=0}^{\infty} \phi_j(\xi)=1,~~\xi \in \Gamma.$

\item[(\romannumeral 3 )] $\bigg|(\widecheck{\phi}_j)^{\langle s \rangle}(x)\bigg| \leq c_s q^{-j+js},~s \in (0, +\infty),~~~j \in \mathbb{N}_0,~~x \in K,$

\ni where $(\widecheck{\phi}_j)^{\langle s \rangle}$ is the  point-wise $s$-order $p$-type derivatives of $\widecheck{\phi}_j$ and $c_s$ is a constant depending on $s$ only.
\end{itemize}

We denote  
$$ A(\Gamma)= \bigg\{ \{\phi_j\}_{j\in \mathbb{N}_0} \subset S(\Gamma)~\text{satisfying}~ (\romannumeral 1) - (\romannumeral 3)\bigg\}.$$
\end{definition}

Specially, for a local field $K$, we just need to take $\{\phi_j\}_{j=0}^{\infty} \subset A(\Gamma)$ as follows:
\begin{align*}
\phi_0(\xi) & = \Phi_{\Gamma^0}(\xi),\\
\phi_j(\xi) & = \Phi_{\Gamma^j \setminus \Gamma^{j-1}}(\xi),~~~j \in \mathbb{N},
\end{align*}
where $\Phi_{A}$ is the characteristic function of $A$. It satisfies the following conditions
\begin{itemize}
\item[(\romannumeral 1)] $\text{supp}~\Phi_{\Gamma^0} \subset \{\xi \in \Gamma  : |\xi|<q\}$,\\
$\text{supp}~\Phi_{\Gamma^j \setminus \Gamma^{j-1}} \subset \{\xi \in \Gamma : q^{j-1} < |\xi|< q^{j+1}\},~~~~j \in \mathbb{N},$\\

\item[(\romannumeral 2)] $ \Phi_{\Gamma^0}(\xi) + \sum_{j=1}^{\infty} \Phi_{\Gamma^{j}\setminus \Gamma^{j-1}}(\xi)=1,~~~~\xi \in \Gamma,$\\

\item[(\romannumeral 3 )] $\bigg|(\widecheck{\Phi}_{\Gamma^0})^{\langle s \rangle}(x)\bigg| \leq c_s,$\\
$\bigg|(\widecheck{\Phi}_{\Gamma^{j}\setminus \Gamma^{j-1}})^{\langle s \rangle}(x)\bigg| \leq c_s q^{-j+js},~s \in (0, +\infty),~~~j \in \mathbb{N},~~x \in K,$\\~\text{where}~$c_s$~\text{is a constant depending on} $s$ \text{only}.
\end{itemize}

\begin{definition} \label{def1}
Let $\{ \phi_j\}_{j=0}^{\infty} = \{\Phi_{\Gamma^0},~\Phi_{\Gamma^{j}\setminus \Gamma^{j-1}}\}_{j=1}^{\infty} \subset A(\Gamma),$ we define the operators $\Delta_j$ as follows
\begin{align*}
\Delta_j f =  \mathcal{F}^{-1} \Big(\phi_j \mathcal{F}f\Big),~~~~~~~~~~j \in \mathbb{N}_0,~~f \in S^{\prime}(K).
\end{align*}

Then we obtain the Littlewood-Paley decomposition,

$$f = \sum_{j=0}^{\infty} \Delta_j f,$$

of all $f \in S^{\prime}(K).$
\end{definition}

Note that $\Delta_j f \in S^{\prime}(K)$ for any $f \in S^{\prime}(K)$. 

We are providing the proof of convergence of Littlewood-Paley decomposition on local fields.

\begin{proposition} 
Let $u$ be in $S^{\prime}(K).$ Then, we have, $u = \sum_{j=0}^{\infty} \Delta_j u,$ in the sense of the convergence in the space $S^{\prime}(K).$
\end{proposition}

\begin{proof} Let $u \in S^{\prime}(K)$ and 
\begin{align*}
S_n u =\sum_{j=0}^{n}\Delta_{j} u.
\end{align*}
 We have 
\begin{align*}
< (I - S_n) u,~f >~ = ~ <u,~(I -S_n) f>,~~~~\text{for all}~~f \in S(K).
\end{align*} 

Thus, it is enough to prove that in the space $S(K),$ we have, 

\begin{align*}
f = \lim_{n \to \infty} S_n f.
\end{align*} 

For all $\xi \in K,$ we have
\begin{align*}
\mathcal{F}(S_n f -f)(\xi) &= \mathcal{F}\bigg(\sum_{j=0}^{n}\Delta_{j} f -f\bigg)(\xi)\\
 & = \sum_{j=0}^{n} \mathcal{F}(\Delta_{j} f)(\xi)-\mathcal{F}(f)(\xi)\\
 &=  \sum_{j=0}^{n} \phi_j \mathcal{F}(f)(\xi) - \mathcal{F}(f)(\xi)\\
 & = \bigg(\sum_{j=0}^{n} \phi_j -1 \bigg) \mathcal{F}(f)(\xi). 
\end{align*}

Since $\sum\limits_{j=0}^{\infty} \phi_j=1.$ Hence
\begin{align*}
 \lim_{n \to \infty} \mathcal{F}(S_n f -f)(\xi) = 0,
\end{align*}
and therefore,
\begin{align*}
f = \lim_{n \to \infty} S_n f.
\end{align*}
\end{proof}

\begin{definition} \textbf{(B-type space)}  
For $0 < r \leq +\infty,~0 < t \leq +\infty,~ s \in \mathbb{R},$ we define B-type spaces or Besov spaces on local fields $K$ as

$$ B_{rt}^s(K)=\{f \in S^{\prime}(K):~ \|f\|_{{ B_{rt}^s}(K)} < \infty\}, $$

\ni with norm

\begin{align*}
\|f\|_{{ B_{rt}^s}(K)} &= \|q^{sj} \Delta_j f \| _{l_t(L^r(K))}\\
&= \Bigg\{\sum\limits_{j=0}^{\infty} \|q^{sj} \Delta_j f\|^t_{L^r(K)}   \Bigg\}^{\frac{1}{t}}\\
&= \Bigg\{\sum_{j=0}^{\infty} q^{sjt} \bigg\{ \int_{K} |\Delta_j f|^r dx \bigg\}^{\frac{t}{r}}  \Bigg\}^{\frac{1}{t}}.
\end{align*}
\end{definition}

\begin{remark} The spaces $B_{rt}^s(K)$ are independent of the selection of sequences  $\{\phi_j\}_{j\in \mathbb{N}_0} \subset A(\Gamma).$ They are quasi-Banach spaces (Banach spaces for $r,t \geq 1$), and $ S(K) \subset B_{rt}^s(K) \subset S^{\prime}(K),$ where the first embedding is dense if $r < \infty$ and $t < \infty.$ The theory of the spaces $B_{rt}^s(K)$  has been developed  in \cite{book_impulse5}.
\end{remark}

Note that members of $B_{rt}^s(K)$ are tempered distributions and which can only be interpreted as regular distributions for sufficiently high smoothness. More
precisely, we have

\begin{align}
B_{rt}^s(K) \subset L^{\text{loc}}_1(K)~~~\text{if and only if}~~~ s > \sigma_r~~\text{for}~~0 < r \leq \infty,~0< t \leq \infty,\ \nonumber
\end{align}

\ni see \citep[Theorem 4.1.3]{book_impulse5}. Since for $0 < s < \sigma_r$ the $\delta$-distribution belongs to $B_{rt}^s(K)$, which is a singular distribution, so, in general one cannot interpret $f \in B_{rt}^s(K)$ as a regular distribution.

\begin{definition} Let $\Omega$ be a compact subset of $\Gamma$ and $0 < r \leq \infty$, then we define
\begin{align*}
S^{\Omega}(K) &= \{ \phi~:~ \phi \in S(K),~ \text{supp}~\mathcal{F}\phi \subset \Omega\}\\
 L_r^{\Omega}(K) &= \{ \psi~:~ \psi \in S^{\prime}(K),~ \text{supp}~\mathcal{F}\psi \subset \Omega,~ \|\psi~ |~ L_r\| < \infty\}.
\end{align*} 
\end{definition}

If $s$ is a real number, then
\begin{align*}
H^s_2 = \{ f~:~ f \in S^{\prime}(K),~~\|f|H^s_2\|= \|\langle \cdot \rangle^s (\mathcal{F}f)(\cdot) | L_2\| < \infty\}.
\end{align*}

By the well-known fact that $\mathcal{F}$ is a unitary operator on $L_2$, we have 
\begin{align*}
\|f~|~H^s_2\|= \|\mathcal{F}^{-1}\langle \cdot \rangle^s \mathcal{F}f ~|~ L_2\|.
\end{align*}

We will be using the following theorem in further sections. We refer \citep[Theorem 1.1.5]{book_impulse9} to the proof of the theorem.
 
\begin{theorem} \label{thm1} Let $\Omega$ be a compact subset of $K$ and $0 < r \leq \infty$. If $s > \bigg(\dfrac{1}{\text{min}(r,1)}-\dfrac{1}{2}\bigg)$ then there exists a constant $c$ such that
\begin{equation}
\|\mathcal{F}^{-1} M \mathcal{F}f ~|~ L_r\| \leq c \|M~|~H^s_2\|~ \|f ~|~ L_r\|,
\end{equation}
 
\ni holds for all $f \in L_r^{\Omega}(K) $ and all $M \in H^s_2.$ 
\end{theorem}


\section{Dilation Operator}

Our main result is the following.

\begin{theorem} \label{thm2} Let $0 < r ,t \leq \infty $, $ s > \sigma_r = \text{max}\big(\frac{1}{r} -1 , 0\big), ~k \in \mathbb{N}$ and $T_k$ defined by (\ref{eq1}). Then 
\begin{equation}
\|T_kf ~ \big| ~{ B_{rt}^s}(K) \| \leq c k^{1/t} q^{k(s-\frac{1}{r})} \|f ~\big|~ { B_{rt}^s}(K) \|,
\end{equation}

\ni for some $c$ which is independent of $k$ and for all $f \in  B_{rt}^s(K).$
\end{theorem}

\begin{proof}
Recall from Definition \ref{def2}, 
\begin{equation*}
\{\phi_j\}_{j=0}^{+\infty} = \{ \Phi_{\Gamma^0},~\Phi_{\Gamma^j \setminus \Gamma^{j-1}}\}_{j=1}^{+\infty}~ \subset~ A(\Gamma),
\end{equation*}
\ni the non-homogeneous unit decomposition on $\Gamma.$ We may assume that 
\begin{equation*}
\phi_j(\xi) = \phi_1(\mathfrak{p}^{j-1} \xi),~~~~~~~ j \in \mathbb{N}.
\end{equation*}

We have 
\begin{align*} 
(\mathcal{F}f(\mathfrak{p}^{-k}\cdot))(\xi) & = \int_{K} f(\mathfrak{p}^{-k}x)\bar{\chi}_\xi(x)dx \\
& = \int_{K} f(t)\bar{\chi}_\xi(\mathfrak{p}^{k}t) q^{-k}dt \\ 
& = q^{-k} \int_{K} f(t)\bar{\chi}_{\mathfrak{p}^{k}\xi}(t) dt \\ 
& = q^{-k} (\mathcal{F}f)(\mathfrak{p}^{k}\xi). 
\end{align*}

Also
\begin{align} \nonumber
\mathcal{F}^{-1}\{\phi_j(\xi)(\mathcal{F}f(\mathfrak{p}^{-k}\cdot))(\xi)\}(x) &= q^{-k} \mathcal{F}^{-1}\{\phi_j(\xi)(\mathcal{F}f)(\mathfrak{p}^{k}\xi)\}(x) \\ \nonumber 
& = q^{-k} \int_{\Gamma} \phi_j(\xi)(\mathcal{F}f)(\mathfrak{p}^{k}\xi)\chi_x(\xi)d\xi \\ \nonumber 
& = q^{-k} \int_{\Gamma} \phi_j(\mathfrak{p}^{-k}\xi)(\mathcal{F}f)(\xi)\chi_x(\mathfrak{p}^{-k}\xi) q^k d\xi \\ \nonumber
& = q^{-k} q^k \int_{\Gamma} \phi_j(\mathfrak{p}^{-k}\xi)(\mathcal{F}f)(\xi)\chi_{\mathfrak{p}^{-k}x}(\xi)  d\xi \\ \label{eqn4}
& = \mathcal{F}^{-1}\{\phi_j(\mathfrak{p}^{-k}\xi)(\mathcal{F}f)(\xi)\}(\mathfrak{p}^{-k} x).\\ \nonumber
\end{align}

From the definition of Besov spaces with $f(\mathfrak{p}^{-k} x)$ in place of $f(x)$ and using (\ref{eqn4}), we obtain
\begin{align} \nonumber
\|f(\mathfrak{p}^{-k} x) ~ \big| ~{ B_{rt}^s}(K) \| &= \Bigg(\sum_{j=0}^{\infty} q^{sjt}\| \mathcal{F}^{-1}\{\phi_j(\xi)(\mathcal{F}f(\mathfrak{p}^{-k}\cdot))(\xi)\}(x) ~ \big|~L_r(K) \|^t   \Bigg)^{\frac{1}{t}} \\ \nonumber
&= \Bigg(\sum_{j=0}^{\infty} q^{sjt}\| \mathcal{F}^{-1}\{\phi_j(\mathfrak{p}^{-k}\xi)(\mathcal{F}f)(\xi)\}(\mathfrak{p}^{-k} x) ~ \big|~L_r(K) \|^t   \Bigg)^{\frac{1}{t}} \\ \label{eqn5}
& =q^{-k/r} \Bigg(\sum_{j=0}^{\infty} q^{sjt}\| \mathcal{F}^{-1}\phi_j(\mathfrak{p}^{-k} \cdot)\mathcal{F}f ~ \big|~L_r(K) \|^t   \Bigg)^{\frac{1}{t}}. \\ \nonumber
\end{align}

If $j \geq k+1,$ then $\phi_j(\mathfrak{p}^{-k} \xi) = \phi_1(\mathfrak{p}^{j-k-1} \xi) = \phi_{j-k}(\xi)$. This gives
\begin{align} \nonumber
q^{-k/r} & \Bigg(\sum_{j=k+1}^{\infty} q^{sjt}\| \mathcal{F}^{-1}\phi_j(\mathfrak{p}^{-k} \cdot)\mathcal{F}f ~ \big|~L_r(K) \|^t   \Bigg)^{\frac{1}{t}} \\ \nonumber
& = q^{-k/r}  \Bigg(\sum_{j=k+1}^{\infty} q^{s(j-k)t} q^{skt} \| \mathcal{F}^{-1}\phi_{j-k}\mathcal{F}f ~ \big|~L_r(K) \|^t   \Bigg)^{\frac{1}{t}} \\ \nonumber
& = q^{-k/r} q^{sk}  \Bigg(\sum_{l=1}^{\infty} q^{slt}  \| \mathcal{F}^{-1}\phi_{l}\mathcal{F}f ~ \big|~L_r(K) \|^t   \Bigg)^{\frac{1}{t}} \\ \nonumber
& = q^{k(s-1/r)}  \Bigg(\sum_{l=1}^{\infty} q^{slt}  \| \mathcal{F}^{-1}\phi_{l}\mathcal{F}f ~ \big|~L_r(K) \|^t   \Bigg)^{\frac{1}{t}} \\ \label{eqn6}
& \leq cq^{k(s-\frac{1}{r})} \|f ~ \big| ~{ B_{rt}^s}(K) \| . \\ \nonumber
\end{align}

For the remaining terms, $j = 0, 1, \cdots, k$ we use Theorem \ref{thm1} and 
\begin{align} \nonumber
\mathcal{F}^{-1}\phi_j(\mathfrak{p}^{-k} \cdot)\mathcal{F}f & = \mathcal{F}^{-1}\phi_j(\mathfrak{p}^{-k} \cdot) \phi_0\mathcal{F}f \\  \label{eqn7}
& = \mathcal{F}^{-1}\phi_j(\mathfrak{p}^{-k} \cdot) \mathcal{F}(\mathcal{F}^{-1}\phi_0\mathcal{F}f). \\ \nonumber
\end{align}

We put $h = \mathcal{F}^{-1}\phi_0\mathcal{F}f $, where $\text{supp}~\mathcal{F}h \subset \text{supp}~\phi_0 = \Gamma^0.$ If $j=0,$ we take $M = \phi_0(\mathfrak{p}^{-k} \cdot) $ and calculate 
\begin{align} \nonumber
q^{-k/r} \| \mathcal{F}^{-1}\phi_0(\mathfrak{p}^{-k}\cdot)\mathcal{F}f ~ \big|~L_r(K) \| & = q^{-k/r} \|\mathcal{F}^{-1}\phi_0(\mathfrak{p}^{-k} \cdot) \mathcal{F}(\mathcal{F}^{-1}\phi_0\mathcal{F}f) ~ \big|~L_r(K) \| \\ \label{eqn8}
& \leq c q^{-k/r} \| \phi_0(\mathfrak{p}^{-k} \cdot) ~ \big|~ H_2^{\sigma} \|~\cdot~ \| \mathcal{F}^{-1}\phi_0\mathcal{F}f ~ \big|~L_r(K) \|,\\ \nonumber
\end{align} 
\ni where $\sigma$ is an arbitrary number with $\sigma > \bigg(\dfrac{1}{\text{min}(r,1)}-\dfrac{1}{2}\bigg)$. It is easy to see that 
\begin{align*}
\| \phi_0(\mathfrak{p}^{-k} \cdot) ~ \big|~ H_2^{\sigma} \| & = \| \langle \xi \rangle^{\sigma} (\mathcal{F}\phi_0(\mathfrak{p}^{-k} \cdot))(\xi) ~ \big|~ L_2(K) \| \\ 
& = q^{-k} \| \langle \xi \rangle^{\sigma} (\mathcal{F}\phi_0)(\mathfrak{p}^{k}\xi) ~ \big|~ L_2(K) \| \\ 
& = q^{-k} q^{k/2} \| \langle \mathfrak{p}^{-k}\xi \rangle^{\sigma} (\mathcal{F}\phi_0)(\xi) ~ \big|~ L_2(K) \| \\ 
& = q^{-k/2} q^{k\sigma} \|(\mathcal{F}\phi_0)(\xi) ~ \big|~ L_2(K) \| \\ 
& = q^{-k/2} q^{k\sigma} \|\phi_0 ~ \big|~ L_2(K) \| \\
& =  q^{k(\sigma-1/2)}.
\end{align*}

We may assume that $s > \sigma-1/2.$ This gives 
\begin{align} \nonumber 
q^{-k/r} \| \mathcal{F}^{-1}\phi_0(\mathfrak{p}^{-k} \cdot)\mathcal{F}f ~ \big|~L_r(K) \| & \leq c q^{-k/r} q^{ks}  \| \mathcal{F}^{-1}\phi_0\mathcal{F}f ~ \big|~L_r(K) \| \\ \label{no1}
& \leq c^{\prime} q^{k(s-\frac{1}{r}}) \|f ~ \big| ~{ B_{rt}^s}(K) \|. \\ \nonumber
\end{align} 

Finally, it remains to consider $1 \leq j \leq k$. This is the crucial step leading to $k^{1/t}.$ In this case, $\phi_j(\mathfrak{p}^{-k}\xi) = \phi_1(\mathfrak{p}^{j-k-1} \xi)$ and 

\begin{align} \nonumber
\mathcal{F}^{-1}\phi_j(\mathfrak{p}^{-k} \cdot)\mathcal{F}f & = \mathcal{F}^{-1}\phi_1(\mathfrak{p}^{j-k-1} \cdot) \phi_0\mathcal{F}f \\ \label{eqn11}
& = \mathcal{F}^{-1}\phi_1(\mathfrak{p}^{j-k-1} \cdot) \mathcal{F}(\mathcal{F}^{-1}\phi_0\mathcal{F}f). \\ \nonumber
\end{align}

We put $h = \mathcal{F}^{-1}\phi_0\mathcal{F}f $, where $\text{supp}~\mathcal{F}h \subset \text{supp}~\phi_0 = \Gamma^0$, and  we take $M = \phi_1(\mathfrak{p}^{j-k-1} \cdot) $. This gives 
\begin{align} \nonumber
\| \mathcal{F}^{-1}\phi_j(\mathfrak{p}^{-k} \cdot)\mathcal{F}f ~ \big|~L_r(K) \| & =  \|\mathcal{F}^{-1}\phi_1(\mathfrak{p}^{j-k-1} \cdot) \mathcal{F}(\mathcal{F}^{-1}\phi_0\mathcal{F}f) ~ \big|~L_r(K) \| \\ \label{eqn12}
& \leq c  \| \phi_1(\mathfrak{p}^{j-k-1} \cdot) ~ \big|~ H_2^{\sigma} \|~\cdot~ \| \mathcal{F}^{-1}\phi_0\mathcal{F}f ~ \big|~L_r(K) \|, \\ \nonumber
\end{align}
\ni where $\sigma$ is an arbitrary number with $\sigma > \bigg(\dfrac{1}{\text{min}(r,1)}-\dfrac{1}{2}\bigg)$ and 
\begin{align} \nonumber
\| \phi_1(\mathfrak{p}^{j-k-1} \cdot) ~ \big|~ H_2^{\sigma} \| & = \| \langle \xi \rangle^{\sigma} (\mathcal{F}\phi_1(\mathfrak{p}^{j-k-1} \cdot)(\xi) ~ \big|~ L_2(K) \| \\ \nonumber
& = q^{-(k-j+1)} \| \langle \xi \rangle^{\sigma} (\mathcal{F}\phi_1)(\mathfrak{p}^{(k-j+1)}\xi) ~ \big|~ L_2(K) \| \\ \nonumber
& = q^{-(k-j+1)} q^{(k-j+1)/2} \| \langle \mathfrak{p}^{(j-k-1)}\xi \rangle^{\sigma} (\mathcal{F}\phi_1)(\xi) ~ \big|~ L_2(K) \| \\ \nonumber
& = q^{-(k-j+1)/2} q^{(k-j+1)\sigma} \|(\mathcal{F}\phi_0)(\xi) ~ \big|~ L_2(K) \| \\ \nonumber
& = q^{-(k-j+1)/2} q^{(k-j+1)\sigma} \|\phi_1 ~ \big|~ L_2(K) \| \\ \label{eqn13}
& =  (q-1)^{1/2}q^{(k-j+1)(\sigma-1/2)}. \\ \nonumber
\end{align}

Using (\ref{eqn12}) and (\ref{eqn13}), we obtain
\begin{align} \nonumber
q^{-k/r} & \Bigg(\sum_{j=1}^{k} q^{sjt}\| \mathcal{F}^{-1}\phi_j(\mathfrak{p}^{-k} \cdot)\mathcal{F}f ~ \big|~L_r(K) \|^t   \Bigg)^{\frac{1}{t}} \\ \nonumber
& \leq c q^{-k/r}  \Bigg(\sum_{j=1}^{k} q^{sjt} q^{(k-j)(\sigma-1/2)t}(q-1)^{t/2}q^{(\sigma-1/2)t}\| \mathcal{F}^{-1}\phi_0\mathcal{F}f ~ \big|~L_r(K) \|^t  \Bigg)^{\frac{1}{t}} \\ \nonumber
& \leq c q^{-k/r}  \Bigg(\sum_{j=1}^{k} q^{sjt} q^{(k-j)st}(q-1)^{t/2}q^{st}\| \mathcal{F}^{-1}\phi_0\mathcal{F}f ~ \big|~L_r(K) \|^t   \Bigg)^{\frac{1}{t}}~~~~~~~~(\because~s > \sigma-1/2) \\ \nonumber
& = c q^{-k/r}  \Bigg(\sum_{j=1}^{k} q^{skt} (q-1)^{t/2}q^{st}\| \mathcal{F}^{-1}\phi_0\mathcal{F}f ~ \big|~L_r(K) \|^t   \Bigg)^{\frac{1}{t}} \\ \nonumber
& = c q^{-k/r} q^{sk} (q-1)^{1/2}q^{s} \| \mathcal{F}^{-1}\phi_0\mathcal{F}f ~ \big|~L_r(K) \| \Bigg(\sum_{j=1}^{k} 1 \Bigg)^{\frac{1}{t}} \\ \label{eqn14} 
& \leq c_1 q^{k(s-1/r)} k^{1/t} \|f ~ \big| ~{ B_{rt}^s}(K) \|. \\ \nonumber 
\end{align}

Finally, (\ref{eqn14}) together with (\ref{no1}),  (\ref{eqn6}) and  (\ref{eqn5}) gives the estimate i.e, 
\begin{align*} 
\|f(\mathfrak{p}^{-k} x) ~ \big| ~{ B_{rt}^s}(K) \| \leq c k^{1/t}q^{k(s-\frac{1}{r})} \|f ~\big|~ { B_{rt}^s}(K) \|,
\end{align*}
\ni which is the desired result.
\end{proof}

\section{Localization Property}
For each $i \in \mathbb{Z},$ we choose an element $z^{k,i} \in K,~k \in \mathbb{N}_0,$ so that the subsets $\mathfrak{P}^{k,i} = z^{k,i}+ \mathfrak{P}^i \subset K$ satisfy $\mathfrak{P}^{k,i} \cap \mathfrak{P}^{l,i} = \emptyset$ if $k \neq l$ and $\cup_{k=0}^{\infty} \mathfrak{P}^{k,i} = K,$ moreover, we choose, $z^{0,i}$ so that $\mathfrak{P}^{0,i}= \mathfrak{P}^{i}.$

Let $f \in S^\prime(K)$ with $\text{supp}~f \subset \mathfrak{P}^i,$ and let
\begin{equation} \label{1}
f^j(x) =  f(\mathfrak{p}^{-j}(x-z^{k,i})),~~~~~~ j \in \mathbb{N},~ x \in K. 
\end{equation}

Of course, $f^j(x),~j \in \mathbb{N}$ have mutually disjoint supports.

\begin{theorem} \label{thm3}
Let $0 < t \leq r \leq \infty,~ s > \sigma_r= \text{max}\big(\frac{1}{r} -1 , 0\big)$ and $f^j(x)$ is defined by (\ref{1}). Then  
\begin{equation} \label{2}
\bigg(\sum_{k=0}^{\infty} \|  f^j(x) ~ |~ B^s_{rt}(K) \|^{r} \bigg)^{1/r}  \leq  cj^{1/t} q^{j(s-\frac{1}{r})} \|f ~\big|~ { B_{rt}^s}(K)\|,
\end{equation}

\ni for some $c$ which is independent of $k$ and for all $f \in B_{rt}^s(K).$
\end{theorem}

To prove the Theorem (\ref{thm3}), we need the following proposition. 
\begin{proposition} \label{p1}
Let $ s > \sigma_r= \text{max}\big(\frac{1}{r} -1 , 0\big),~0 < r,t \leq \infty$ and  $\{u_j\}_{j=0}^{\infty}$ be a sequence in $S^\prime(K)$ satisfying 
\begin{align*}
\text{supp}~ \mathcal{F}u_0 \subset \Gamma^0,~~~\text{supp}~ \mathcal{F}u_j \subset \Gamma^j \setminus \Gamma^{j-1}~~\text{for}~ j \in \mathbb{N}.
\end{align*}

If $\| \{q^{sj} \| u_j ~ |~ {L^r(K)} \|\}~ | ~ l^t\|$ is finite. Then there exists $c > 0$ such that 
\begin{equation} \label{3}
\bigg\|\sum_{j=0}^{\infty} u_j~ |~ B^s_{rt}(K) \bigg\|~ \leq~ c   \Bigg\{\sum_{j=0}^{\infty} q^{sjt} \| u_j ~ |~ {L^r(K)} \|^t  \Bigg\}^{\frac{1}{t}}.
\end{equation} 
\end{proposition}

\begin{proof}
Let $\{u_j\}_{j=0}^{\infty}$ be a sequence in $S^\prime(K)$ satisfying $\text{supp}~ \mathcal{F}u_0 \subset \Gamma^0$ and $\text{supp}~ \mathcal{F}u_j \subset \Gamma^j \setminus \Gamma^{j-1}~~\text{for}~ j \in \mathbb{N}.$  Then, we have
\begin{align*}
\Delta_k \bigg(\sum_{j=0}^{\infty} u_j \bigg) & = \sum_{j=0}^{\infty} \Delta_k u_j \\
& =   \Delta_k u_k, 
\end{align*}

\ni since $\text{supp}~ \mathcal{F}(\Delta_k u_j) \neq \emptyset $ only if $j=k.$ By hypothesis $s > \text{max}\big(\frac{1}{r} -1 , 0\big),$ we can choose a real number $\lambda$ satisfying $s > \lambda - 1/2 > \text{max}\big(\frac{1}{r} -1 , 0\big).$ Then by Theorem \ref{thm1}, we have

\begin{equation} \label{11}
\|\Delta_ku_k~|~L_r\| \leq c \|\phi_k~|~H^\lambda_2\|~ \|u_k~|~ L_r\|, 
\end{equation}
and 
\begin{equation} \label{12}
\|\phi_k~|~H^\lambda_2\| \leq c~ q^{-(k-1)(\lambda - \frac{1}{2})}.
\end{equation}
Therefore,
\begin{align*} 
 \| \Delta_k \bigg(\sum_{j=0}^{\infty} u_j \bigg)~ |~ L^r(K)\| ~& = \| \Delta_k u_k~ |~ L^r(K) \| \\
& \leq~ c q^{-(k-1)(\lambda - \frac{1}{2})} \|  u_k~ |~ L^r(K) \|.
\end{align*}
Hence,
\begin{align*}
\Big(\sum_{k=0}^{\infty} q^{kst} \| \Delta_k \bigg(\sum_{j=0}^{\infty} u_j \bigg)~ |~ L^r(K)\|^t \Big)^{1/t} ~ & \leq ~ c^{\prime}   \bigg( \sum_{k=0}^{\infty} q^{k(s-(\lambda - \frac{1}{2}))t} \|  u_k~ |~ L^r(K) \|^t \bigg)^{1/t} \\
& \leq~ c^{\prime} \bigg(\sum_{k=0}^{\infty} q^{kst} \|  u_{k}~ |~ L^r(K) \|^t \bigg)^{1/t}.
\end{align*}
\end{proof}
\begin{proof}[Proof of Theorem \ref{thm3}]

Let $s > \sigma_r~\text{and}~ 0 < t \leq r \leq \infty.$ Then, by Theorem \ref{thm2}, we have
\begin{equation} \label{17}
\|f(\mathfrak{p}^{-j}(x-z^{k,i})) ~ \big| ~{ B_{rt}^s}(K) \|~ \leq ~c j^{1/t} q^{j(s-\frac{1}{r})} \|f(x-z^{k,i}) ~\big|~ { B_{rt}^s}(K) \|.
\end{equation}
  
Therefore in order to complete the proof of Theorem \ref{thm3} it is enough to show that
\begin{equation} \label{16}
\bigg(\sum_{k=0}^{\infty} \|  f(x-z^{k,i}) ~ |~ B^s_{rt}(K) \|^{r} \bigg)^{1/r} ~\leq~ c\|f ~\big|~ { B_{rt}^s}(K) \|.
\end{equation}

To prove (\ref{16}), observe the following inequality which will follow by Proposition \ref{p1},
\begin{align*}
\Bigg\| \sum_{j=0}^{\infty} \Delta_j f(x-z^{k,i})~ \big|~ B^s_{rt}(K) \Bigg\|~ \leq~ c \bigg(\sum_{j=0}^{\infty} q^{sjt}\|  \Delta_j f(x-z^{k,i}) ~ |~ {L^r(K)}\|^{t} \bigg)^{1/t}.
\end{align*}

Since
\begin{align*}
\text{supp}~ \mathcal{F}(\Delta_0 f(x-z^{k,i})) & = \text{supp}~(\phi_0\mathcal{F}f) \\
& \subset \text{supp}~\phi_0 \subset \Gamma^0,
\end{align*}
\ni and
\begin{align*}
\text{supp}~ \mathcal{F}(\Delta_j f(x-z^{k,i})) & = \text{supp}~(\phi_j\mathcal{F}f) \\
& \subset \text{supp}~\phi_j \subset \Gamma^j \setminus \Gamma^{j-1} ,~~~~j \in \mathbb{N}.
\end{align*}

This implies that,
\begin{align*}
\|f(x-z^{k,i})~ \big|~ B^s_{rt}(K)\|^r ~ \leq ~ c \bigg(\sum_{j=0}^{\infty} q^{sjt}\|  \Delta_j f(x-z^{k,i}) ~ |~ {L^r(K)}\|^{t} \bigg)^{r/t}.
\end{align*}

Then it holds that
\begin{equation} \label{4}
 \bigg(\sum_{k=0}^{\infty} \|  f(x-z^{k,i}) ~ |~ B^s_{rt}(K) \|^{r} \bigg)^{1/r}~ \leq~ c \Bigg(\sum_{k=0}^{\infty}\bigg(\sum_{j=0}^{\infty} q^{sjt}\|  \Delta_j f(x-z^{k,i})~ |~ {L^r(K)}\|^{t} \bigg)^{r/t}\Bigg)^{1/r}.
\end{equation}

Since, $r \geq t$ which implies that $r/t \geq 1$. Then by using triangle inequality
\begin{align*}
\|(q^{sjt}\|\Delta_j f(x-z^{k,i})\|_{L^r})_{k \in \mathbb{N}_0}\|_{l^{r}(l^t)} & = \Bigg(\sum_{k=0}^{\infty}\bigg(\sum_{j=0}^{\infty} q^{sjt}\|  \Delta_j f(x-z^{k,i})~ |~ {L^r(K)}\|^{t} \bigg)^{r/t}\Bigg)^{1/r}\\
 & \leq ~  \Bigg(\sum_{j=0}^{\infty}q^{sjt}\bigg(\sum_{k=0}^{\infty} \|  \Delta_j f(x-z^{k,i})~ |~ {L^r(K)}\|^{t\times \frac{r}{t}} \bigg)^{t/r}\Bigg)^{\frac{1}{r}\times \frac{r}{t}} \\
 & = ~  \Bigg(\sum_{j=0}^{\infty}q^{sjt}\bigg(\sum_{k=0}^{\infty} \|  \Delta_j f(x-z^{k,i})~ |~ {L^r(K)}\|^{r} \bigg)^{t/r}\Bigg)^{\frac{1}{t}} \\
& = ~  \Bigg(\sum_{j=0}^{\infty}q^{sjt}\|  \Delta_j f(x-z^{k,i})~ |~ l^r(L^r(K))\|^{t}\Bigg)^{\frac{1}{t}}.
\end{align*}

So, the inequality (\ref{4}) becomes 
\begin{equation} \label{5}
 \bigg(\sum_{k=0}^{\infty} \|  f(x-z^{k,i}) ~ |~ B^s_{rt}(K) \|^{r} \bigg)^{1/r}~ \leq~ c \Bigg(\sum_{j=0}^{\infty}q^{sjt}\|  \Delta_j f(x-z^{k,i})~ |~ l^r(L^r(K))\|^{t}\Bigg)^{\frac{1}{t}}.
\end{equation}

Since $\|  \Delta_j f(x-z^{k,i})~ |~ l^r(L^r(K))\|~ \sim~ \|  \Delta_j f~ |~ {L^r(K)}\|.$ Hence

\begin{align} \nonumber
\bigg(\sum_{k=0}^{\infty} \|  f(x-z^{k,i}) ~ |~ B^s_{rt}(K) \|^{r} \bigg)^{1/r}~ & \leq~ c \Bigg(\sum_{j=0}^{\infty}q^{sjt}\|  \Delta_j f(x-z^{k,i})~ |~ l^r(L^r(K))\|^{t}\Bigg)^{\frac{1}{t}} \\ \nonumber
& \leq~ c \Bigg(\sum_{j=0}^{\infty}q^{sjt}\|  \Delta_j f~ |~ L^r(K)\|^{t}\Bigg)^{\frac{1}{t}} \\ \label{6}
& =~ c \|f ~\big|~ { B_{rt}^s}(K) \|. \\ \nonumber
\end{align}

Now (\ref{6}) together with (\ref{17}) gives
\begin{align*}
\bigg(\sum_{k=0}^{\infty} \|  f^j(x) ~ |~ B^s_{rt}(K) \|^{r} \bigg)^{1/r}  \leq  cj^{1/t} q^{j(s-\frac{1}{r})} \|f ~\big|~ { B_{rt}^s}(K)\|.
\end{align*}

This completes the proof.
\end{proof}

\bibliographystyle{achemso}

\end{document}